\documentclass[11pt]{article}

\usepackage{amsfonts,amsmath,latexsym,color,epsfig,hyperref}
\newtheorem{theorem}{Theorem}[section]
\newtheorem{lemma}[theorem]{Lemma}

\newtheorem{conjecture}{Conjecture}
\newtheorem{corollary}[theorem]{Corollary}

\newtheorem{claim}{Claim}

\newtheorem{definition}{Definition}

\newenvironment{proof}
      {\medskip\noindent{\bf Proof:}\hspace{1mm}}
      {\hfill$\Box$\medskip}
\def\qed{\ifvmode\mbox{ }\else\unskip\fi\hskip 1em plus 10fill$\Box$}

\input{epsf}

\makeatletter
\def\Ddots{\mathinner{\mkern1mu\raise\p@
\vbox{\kern7\p@\hbox{.}}\mkern2mu
\raise4\p@\hbox{.}\mkern2mu\raise7\p@\hbox{.}\mkern1mu}}
\makeatother

\title{\vspace{-0.7cm}Cycle packing}
\author{David Conlon\thanks{Mathematical Institute, Oxford OX2 6GG, UK. Email: {\tt david.conlon@maths.ox.ac.uk}. Research supported by a Royal Society University Research Fellowship.} \and Jacob Fox\thanks{Department of Mathematics, MIT, Cambridge, MA 02139-4307. Email: {\tt fox@math.mit.edu}. Research supported by a Packard Fellowship, by a Simons Fellowship, by NSF grant DMS-1069197, by an Alfred P. Sloan Fellowship and by an MIT NEC Corporation Award.} \and Benny Sudakov\thanks{
Department of Mathematics, ETH, 8092 Zurich, Switzerland. Email: {\tt benjamin.sudakov@math.ethz.ch}. Research supported in part by SNSF grant 200021-149111 and by a USA-Israel BSF grant.}}
\date{}

\begin{document}
\maketitle

\begin{abstract}
In the 1960s, Erd\H{o}s and Gallai conjectured that the edge set of every graph on $n$ vertices can be partitioned into $O(n)$ cycles and edges. They observed that one can easily get an $O(n \log n)$ upper bound by repeatedly removing the edges of the longest cycle. We make the first progress on this problem, showing that $O(n \log \log n)$ cycles and edges suffice. We also prove the Erd\H{o}s-Gallai conjecture for random graphs and for graphs with linear minimum degree.
\end{abstract}

\section{Introduction}

Packing and covering problems have a rich history in graph theory and many of the oldest and most intensively studied topics in this area (see~\cite{P92}) relate to packings and coverings with paths and cycles. For example, in 1968, Lov\'asz~\cite{L68} proved the following fundamental result about decompositions or edge partitions of a graph into paths and cycles.

\begin{theorem}[Lov\'asz] \label{Lovaszlem}
Every graph on $n$ vertices can be decomposed into at most $\frac{n}{2}$ paths and cycles.
\end{theorem}

Theorem~\ref{Lovaszlem} easily implies that there is a decomposition of any graph on $n$ vertices into $n - 1$ paths. This was subsequently improved~\cite{DK00, Y98} to $\lfloor \frac{2n}{3} \rfloor$ paths, a result which is sharp for a disjoint union of triangles. However, Lov\'asz' original motivation for studying such decompositions, a problem of Gallai which asks whether every connected graph on $n$ vertices can be decomposed into $\lfloor \frac{n+1}{2} \rfloor$ paths, remains open. The following old and well-known conjecture of Erd\H{o}s and Gallai~\cite{E83,EGP66} concerning the analogous question for cycles also remains open.

\begin{conjecture}[Erd\H{o}s-Gallai]
Every graph on $n$ vertices can be decomposed into $O(n)$ cycles and edges.
\end{conjecture}

More progress has been made on the corresponding covering problems, where we no longer insist that the edges of the paths and cycles have to be disjoint. For example, Pyber~\cite{P85} resolved the covering version of the Erd\H{o}s-Gallai conjecture, showing that every graph on $n$ vertices can be covered by $n-1$ cycles and edges. Similarly, settling a question of Chung~\cite{C80}, Fan \cite{F02} proved the covering version of Gallai's conjecture, showing that the edges of every connected graph on $n$ vertices may be covered by $\lfloor \frac{n+1}{2} \rfloor$ paths (an asymptotic version was proved earlier by Pyber~\cite{P96}). However, the decomposition conjectures are thought to be more difficult (see~\cite{C80, P92}).


Let $f(n)$ be the minimum number such that every graph on $n$ vertices can be decomposed into at most $f(n)$ cycles and edges. The Erd\H{o}s-Gallai conjecture states that $f(n)=O(n)$. An example of Gallai (see \cite{EGP66}) shows that $f(n) \geq (\frac{4}{3} - o(1))n$ and Erd\H{o}s \cite{E83} later remarked that there is an example showing that $f(n) \geq (\frac{3}{2}-o(1))n$. As noted in~\cite{EGP66}, it is easy to see that $f(n)=O(n\log n)$. Indeed, Erd\H{o}s and Gallai \cite{EG59} showed that every graph with $n$ vertices and $m > \ell(n-1)/2$ edges contains a cycle of length at least $\ell$. By greedily removing cycles of longest length, we see, after removing $O(n)$ cycles, that the graph that remains will be acyclic or have at most half the edges. The bound $f(n)=O(n\log n)$ follows from a simple iteration. Here we make the first progress on the Erd\H{o}s-Gallai conjecture, showing that $f(n)=O(n \log \log n)$.
This is a corollary of the following stronger result.

\begin{theorem}\label{firstmain}
Every graph on $n$ vertices with average degree $d$ can be decomposed into $O(n \log \log d)$ cycles and edges.
\end{theorem}

We will also prove the Erd\H{o}s-Gallai conjecture in certain special cases. In particular, we may exploit the fact that random graphs are good expanders to prove that they satisfy the Erd\H{o}s-Gallai conjecture. The random graph $G(n,p)$ on vertex set
$[n] = \{1, \dots ,n\}$ is constructed by taking each potential edge independently with probability $p$. We
say that $G(n,p)$ possesses a property $\cal P$ asymptotically almost surely, or a.a.s.~for 
short, if the probability that $G(n,p)$ possesses $\cal P$ tends to $1$ as $n$ grows to infinity.

\begin{theorem}\label{thmrandom}
There exists a constant $c > 0$ such that for any probability $p := p(n)$ the random graph $G(n,p)$ a.a.s.~can be decomposed into at most $cn$ cycles and edges.
\end{theorem}

Building on the ideas used to prove the Erd\H{o}s-Gallai conjecture in random graphs, we also prove the Erd\H{o}s-Gallai conjecture for graphs of linear minimum degree.

\begin{theorem} \label{EGlargemin}
Every graph $G$ on $n$ vertices with minimum degree $cn$ can be decomposed into at most $O(c^{-12}n)$ cycles and edges.
\end{theorem}

We will prove Theorem~\ref{firstmain} in the next section, reserving the proof of a key technical lemma to Section~\ref{firstmainlemma}. We prove Theorem~\ref{thmrandom}, that the Erd\H{o}s-Gallai conjecture holds for random graphs, in Section~\ref{secrandom}. In Section~\ref{sectcutdense}, we show that the Erd\H{o}s-Gallai conjecture holds for graphs with no sparse cut and in Section~\ref{sectmin} we use this result to prove Theorem~\ref{EGlargemin}.
For the sake of clarity of presentation, we systematically omit floor and ceiling signs whenever they are not crucial. We also do not make any serious attempt to optimize absolute constants in our statements and proofs. Unless stated otherwise, we will use $\log$ to denote logarithm taken to the base $2$ and $\ln$ for the natural logarithm.

\section{General graphs}
\label{sectproofs}

In this section and the next, we will prove Theorem~\ref{firstmain}, that every graph on $n$ vertices with average degree $d$ can be decomposed into $O(n \log \log d)$ cycles and edges.

We will begin by stating a couple of useful lemmas. The first is a slight variant of P\'osa's celebrated rotation-extension lemma (\cite{Po76}, see also \cite{B01}, \cite{L07}) which says that if a graph does not contain long paths or cycles then it has poor expansion properties. For a given graph $G$ and a subset of its vertices $X$, the external neighborhood $N(X)$ denotes the set of vertices in $G-X$ which have at least one neighbor in $X$. The following P\'osa-type statement follows immediately from Lemmas 2.6 and 2.7 in \cite{BBDK}.

\begin{lemma}\label{Posa}
If a graph $G$ contains no cycle of length greater than $3t$ then there is a subset $S$ of size at most $t$ such that $|N(S)| \leq 2|S|$.
\end{lemma}

Recall that the circumference of a graph is the length of the longest cycle. The following lemma says that if a graph has small circumference then it may be split into subgraphs of small order whose vertex sets do not overlap by much on average.

\begin{lemma}\label{circumference}
If $G$ is a graph with $n \geq 3$ vertices and circumference at most $t$ then there is an edge partition of $G$ into subgraphs $G_1, \dots, G_s$ such that $|V(G_i)| \leq t + 2$ and $\sum_{i=1}^s |V(G_i)| \leq 3 n - 6$.
\end{lemma}

\begin{proof}
The proof is by induction on $n$. The base case $n = 3$ is trivial. Assume, therefore, that $G$ is a graph on $n \geq 4$ vertices. Then, by Lemma~\ref{Posa}, there is a set $S$ of size at most $\lceil t/3 \rceil$ such that $|N(S)| \leq 2|S|$. Let $G_1$ be the induced subgraph on $S \cup N(S)$, so the number of vertices of $G_1$ is at most $3|S| \leq t + 2$. Let $G'$ be the subgraph on $V\char92 S$ with edge set $E(G)\char92 E(G_1)$ so that $G_1$ and $G'$ form an edge partition of $G$. If $|V(G')| \geq 3$, we apply the induction hypothesis to edge partition $G'$ into subgraphs of order at most $t + 2$ such that the sum of the sizes of their vertex sets is at most $3 |V(G')| - 6$. Together with $G_1$, which also satisfies $|V(G_1)| \leq t + 2$, the total number of vertices used is at most
\[|V(G_1)| + 3|V(G')| - 6 \leq 3|S| + 3|V(G')| - 6 = 3|S| + 3|V\char92 S| - 6 = 3 n - 6.  \]
If $|V(G')| < 3$, then $G_1$ and $G'$ form the desired partition as $|V(G_1)| + |V(G')| \leq n + 2 \leq 3n - 6$.
\end{proof}

Our main lemma, which we will prove in the next section, is as follows.

\begin{lemma} \label{lem:main}
For every graph on $n \geq 2^{2000}$ vertices, it is possible to partition all but at most $n^{2 - \frac{1}{10}}$ edges into at most $\frac{n}{2}$ cycles.
\end{lemma}

This easily implies that for any $n$ (not only those which are sufficiently large) and any graph on $n$ vertices, it is possible to partition all but at most $n^{2 - \frac{1}{10}}$ edges into $O(n)$ cycles. We will apply the result in this form in the proof of the following lemma.

\begin{lemma}\label{almostthere}
If $G$ is a graph with $n$ vertices and circumference at most $t$ then it is possible to delete $O(n)$ cycles such that the average degree of the graph that remains is at most $18 t^{9/10}$.
\end{lemma}

\begin{proof}
If $n < 3$, there is nothing to prove. For $n \geq 3$, apply Lemma~\ref{circumference} to $G$ to find an edge partition of $G$ into subgraphs $G_1, \dots, G_s$ such that $|V(G_i)| \leq t + 2$ and $\sum_{i=1}^s |V(G_i)| \leq 3 n$. We now apply Lemma~\ref{lem:main} to each $G_i$. This tells us that we may partition the edges of $G_i$ into at most $O(|V(G_i)|)$ cycles and at most $|V(G_i)|^{2-\frac{1}{10}}$ remaining edges. The total number of cycles used is at most
\[\sum_{i=1}^s O(|V(G_i)|) = O(n).\]
Moreover, the total number of edges remaining is at most
\[\sum_{i=1}^s |V(G_i)|^{2 - \frac{1}{10}} \leq \left(\max_i |V(G_i)|^{9/10}\right) \cdot \left(\sum_{i=1}^s |V(G_i)|\right) \leq (t + 2)^{9/10} \cdot 3n \leq 9 t^{9/10} n.\]
Hence the average degree is at most $18t^{9/10}$.
\end{proof}

By deleting long cycles and applying Lemma~\ref{almostthere}, we may prove an iteration step which will be sufficient to imply Theorem~\ref{firstmain}.

\begin{lemma}\label{noreallyalmostthere}
If $G$ is a graph with $n$ vertices and average degree $d \geq 30$ then it is possible to delete $O(n)$ cycles so that the graph that remains has average degree at most $d^{9/10}$.
\end{lemma}

\begin{proof}
Repeatedly delete cycles of maximum length until there are no cycles of length greater than $d/30$. Because there are $dn/2$ edges, at most $(dn/2)/(d/30) = 15n$ cycles are deleted in this way. The circumference of what remains is at most $t = d/30 \geq 1$. Applying Lemma~\ref{almostthere}, we can delete $O(n)$ more cycles so that the graph that remains has average degree at most $18 t^{9/10} \leq d^{9/10}$.
\end{proof}

Theorem~\ref{firstmain} now follows by repeated application of Lemma~\ref{noreallyalmostthere}.

\vspace{2mm}

{\bf Proof of Theorem~\ref{firstmain}:} Repeatedly apply Lemma~\ref{noreallyalmostthere}, removing $O(n)$ cycles at each step. After $O(\log \log d)$ steps, the average degree of the remaining graph will be less than $30$. Once the average degree drops below $30$, we decompose the remaining graph into at most $15 n$ edges. Adding up the number of cycles and edges completes the proof.
{\hfill$\Box$\medskip}

\section{Proof of Lemma~\ref{lem:main}} \label{firstmainlemma}

We begin with the following lemma which says that it is possible, after deleting some edges, to partition the vertex set of a graph into components with a certain expansion property. For vertex subsets $X$ and $Y$, let $e(X,Y)$ denote the number of pairs in $X \times Y$ which are edges. If $X=\{x\}$, we will simply write $e(x,Y)$. 

\begin{lemma} \label{lem:expansion}
Given a graph $G$ on $n$ vertices and $s \in \mathbb{N}$, it is possible to delete at most $4 s n \log n$ edges from $G$ so that the remaining subgraph may be partitioned into components $C_1, C_2, \dots, C_r$ such that $e(C_i, C_j) = 0$ for all $i \neq j$ and the following expansion property holds. For all $i$ and all $X \subset C_i$ with $|X| \leq \frac{|C_i|}{2}$, $e(X, X^c) \geq s |X|$.
\end{lemma}

\begin{proof}
As long as some component $C$ has a set $X$ with $|X| \leq \frac{|C|}{2}$ and $e(X, X^c) < s |X|$, we delete all of the edges between $X$ and $X^c$. We continue until this is no longer possible. We now consider the number of edges deleted from those sets $X$ for which
\[\frac{n}{2^{j+1}} < |X| \leq \frac{n}{2^j}.\]
As, for all $j$, each vertex is in at most one such $X$, the number of such $X$ is at most $2^{j+1}$. Therefore, the number of deleted edges is at most
\[\sum_j 2^{j+1} \frac{n}{2^j} s \leq 4 \log n \cdot s n,\]
completing the proof.
\end{proof}

The following technical lemma is the key to our proof. It says that if a graph has the expansion property satisfied by the components in the last lemma then there is a vertex subset $U$ such that any two points in the complement of $U$ may be connected by a short path in $U$. Moreover, this is true even if certain vertices and edges of $U$ are not allowed to be in the path.

\begin{lemma}\label{lem:randomsub}
Let $G$ be a graph with $n$ vertices such that every set $X$, $|X| \leq \frac{n}{2}$ has $e(X, X^c) \geq s |X|$, where $s = 3 n^{8/9}$. Then there exists a set $U$ of order $u = 3 n^{8/9}$ such that, for all $x, y \in V(G) \char92 U$, $U$ has the following property. For all subsets $W$ of $U$ of size at most $2\sqrt{n}$ and all collections $E$ of at most $\frac{1}{2} n^{4/3}$ edges, the graph $G$ has a path from $x$ to $y$ of length at most $n^{2/9}$ all of whose internal vertices are in $U \char92 W$ and where no edge of the path with both endpoints in $U$ is in $E$.
\end{lemma}

\begin{proof}
Let $N_0 =  \{x\}$. Let $N_1 = \{x\} \cup N(x)$. For any $i \geq 1$, let
\[M_i = \{ v \in V\char92 N_i : e(v, N_i) \geq \frac{|N_i|}{2 (n/s)}\} \mbox{ and } N_{i+1} = N_i \cup M_i.\]
Because the $N_i$ are nested, $|N_i| \geq |N_1| \geq s$.
By the definition of $M_i$,
\begin{align*}
e(N_i, M_i) & = e(N_i, N_i^c) - e(N_i, N_i^c \char92 M_i) \\
& \geq e(N_i, N_i^c) - \frac{|N_i| |N_i^c|}{2 (n/s)} \geq \frac{1}{2} e(N_i, N_i^c) \geq \frac{1}{2} \min\{|N_i|, |N_i^c|\} s,
\end{align*}
where we used that $e(N_i, N_i^c) \geq \min\{|N_i|, |N_i^c|\} s$ and $|N_i||N_i^c| \leq n \min\{|N_i|, |N_i^c|\}$.
Recall that $|N_i| \geq s$. Therefore, if $|N_i^c| \geq \frac{s}{2}$ then $e(N_i, M_i) \geq \frac{1}{2} \min\{|N_i|, |N_i^c|\} s \geq \frac{s^2}{4}$ and $|M_i| \geq \frac{e(N_i, M_i)}{|N_i|} \geq \frac{s^2}{4n}$. Otherwise, $M_i = N_i^c$. Indeed, each vertex has at least $s$ neighbors by the expansion property. Hence, each vertex in $N_i^c$ has at least $s - |N_i^c| > \frac{s}{2} \geq \frac{|N_i|}{2 (n/s)}$ neighbors in $N_i$ and hence must be in $M_i$. Since $s \geq s^2/4n$, it follows that the number of steps $i$ before we exhaust all vertices is at most
\[\left\lceil \frac{n}{s^2/4n} \right\rceil = \left\lceil \frac{4}{9} n^{2/9}\right\rceil \leq n^{2/9}.\]

\begin{claim}
There exists a set $U$ of order $u$ such that for all $x$ (which defines all $N_i$), all $i$ and all $z \in M_i = N_{i+1}\char92 N_i$,
\[e(z, B_i) \geq \frac{|N_i|}{4 (n/s)} \cdot \frac{|U|}{n} \geq \frac{s^2 u}{4 n^2} > 4 n^{2/3},\]
where $B_i = U \cap N_i$.
\end{claim}

\begin{proof}
Let $U$ be a vertex subset of order $u$ picked uniformly at random. Fix $x$, $i$ and $z$. Let $\mu = \mathbb{E} [e(z, B_i)] \geq \frac{|N_i|}{2 (n/s)} \cdot \frac{|U|}{n} \geq \frac{s^2 u}{2 n^2}$.
By Chernoff's inequality for the hypergeometric distribution (see \cite{AS, H63}), $e(z, B_i) < \frac{\mu}{2}$ with probability at most $e^{-\mu/8}$. Since each $z$ is in at most one $M_i$, there are at most  $n$ choices of $z$ and $i$. Hence, summing over all choices of $x$, $z$ and $i$, the union bound implies that the claim holds with probability at least $1 - n^2 e^{-\mu/8} \geq 1 - n^2 e^{- 27 n^{2/3}/16} > 0$.
\end{proof}

Fix a subset $W$ of $U$ of size at most $2 \sqrt{n}$ and a collection $E$ of at most $\frac{1}{2} n^{4/3}$ edges. Because $E$ has size at most $\frac{1}{2} n^{4/3}$, at most $n^{2/3}$ vertices are incident with more than $n^{2/3}$ edges of $E$ with both endpoints in $U$. Let $R$ be the set of such vertices and let $T = R \cup W$. Then $T \subset U$ and
$|T| \leq 2 \sqrt{n} + n^{2/3} \leq 3 n^{2/3}$.

We will build a path $y_0 = y, y_1, \dots, y_{\ell} = x$ which avoids $W$ and the edges of $E$ whose endpoints are in $U$ as follows. Note that there exists $j_0$ such that $y \in M_{j_0} \char 92 T$ as $y \not\in T \subset U$. We may choose the vertex $y_{i+1}$ to be in $U \cap M_{j_{i+1}}\char92 T$ where $j_{i+1} < j_i$. We can do this because the number of neighbors of $y_i$ in $B_{j_i}$ is at least $4 n^{2/3}$ and $T$ is at most $3 n^{2/3}$. Since the $j_i$ are always dropping, the process must terminate within $n^{2/9}$ steps.
\end{proof}

The next lemma says that if we have a decomposition of the edge set of a graph into a small number of paths then we also have a decomposition into a small number of paths and edges such that no vertex is the endpoint of too many paths.

\begin{lemma}\label{sparseends}
Suppose a graph $G$ has a decomposition into $h$ paths. Then $G$ can also be decomposed into $h$ subpaths of these paths and at most $2h$ edges so that each vertex is an endpoint of at most $\sqrt{2h}$ of the paths.
\end{lemma}
\begin{proof}
Let $r=\sqrt{2h}$. There are $2h$ endpoints of the $h$ paths. Call a vertex  {\it dangerous} if it is an endpoint of more than $r$ of the paths. Let $v$ be a {\it dangerous} vertex.
Then one of its (more than $r$) neighbors $u$ is an endpoint of at most $(2h-r)/r=r-1$ of the paths.  Delete the edge $(u,v)$ from one of the paths ending in $v$, thus moving the endpoint of this path from $v$ to $u$. Now $u$ is an endpoint of at most $r-1+1=r$ paths and hence still not dangerous. Therefore, by deleting the edge $(u,v)$, we have reduced the number of endpoints of paths which are dangerous.  Repeating, we can get rid of all dangerous endpoints by deleting at most $2h$ edges and obtain the desired partition.
\end{proof}

We may now combine the last two lemmas to show that if a graph has strong expansion properties then it is possible to partition all but $4n^{2 - 1/9}$ edges into at most $\frac{n}{2}$ cycles.

\begin{lemma} \label{lem:keycomp}
Let $G$ be a graph with $n$ vertices such that every set $X$ with $|X| \leq \frac{n}{2}$ has $e(X, X^c) \geq s |X|$, where $s = 3 n^{8/9}$. Then it is possible to partition all but $4 n^{2 - \frac{1}{9}}$ edges of $G$ into at most $\frac{n}{2}$ cycles.
\end{lemma}

\begin{proof}
Apply Lemma~\ref{lem:randomsub} to obtain a set $U$ with the desired properties. Set $U$ aside with all the edges touching it. Note that there are at most $|U| n \leq 3 n^{2 - \frac{1}{9}}$ such edges from $U$. Denote the remaining induced graph on vertex set $V(G)\char92 U$ by $G'$. By applying Theorem~\ref{Lovaszlem} to $G'$, we obtain an edge partition into at most $\frac{n}{2}$ paths and cycles. Set the cycles aside. Consider the paths $P_1, \dots, P_h$ with $h \leq \frac{n}{2}$. Applying Lemma~\ref{sparseends}, we decompose the union of the paths into subpaths $P'_i$ and at most $n$ edges such that each endpoint of $P'_i$ is the endpoint of at most $\sqrt{n}$ other paths $P'_j$.

Let $x_i, y_i$ be the endpoints of $P'_i$. We will find edge-disjoint paths from $x_i$ to $y_i$ with internal vertices in $U$ to close $P'_i$ to a cycle. Suppose that we have already obtained such paths for $j < i$ and we wish to obtain the relevant path for $i$. Delete from $G$ all the edges which were used to close $P'_j$, for all $j < i$, to a cycle. Since each of $x_i$ and $y_i$ are endpoints of at most $\sqrt{n}$ such paths, they each lose  at most $\sqrt{n}$ edges to $U$. Let $W$ be the set of endpoints of these edges in $U$. Then $W$ has size at most $2\sqrt{n}$. Moreover, to close every path $P'_j$ to a cycle, we used at most $n^{2/9}$ edges in $G[U]$. Therefore, we deleted at most $\frac{n}{2}n^{2/9}<\frac{1}{2} n^{4/3}$ such edges. By Lemma~\ref{lem:randomsub}, there is a path of length at most $n^{2/9}$ between $x_i$ and $y_i$ all of whose internal vertices are in $U$. This allows us to close $P'_i$ into a cycle. In the end, we will have covered all but at most $3 n^{2 - \frac{1}{9}} + n \leq 4 n^{2 - \frac{1}{9}}$ edges.
\end{proof}

Lemma~\ref{lem:main} now follows by applying Lemma~\ref{lem:keycomp} to each of the components given by Lemma~\ref{lem:expansion}.

\vspace{2mm}

{\bf Proof of Lemma~\ref{lem:main}:} Apply Lemma~\ref{lem:expansion} with $s = 3 n^{8/9}$ to delete at most $4 s n \log n$ edges from $G$ so that the remaining subgraph may be partitioned into components $C_1, C_2, \dots, C_r$ which have the following expansion property. For all $i$ and all $X \subset C_i$ with $|X| \leq \frac{|C_i|}{2}$, $e(X, X^c) \geq s |X|$. We may therefore apply Lemma~\ref{lem:keycomp} to each $C_i$ to partition all but $4 |C_i|^{2 - 1/9}$ edges into at most $\frac{|C_i|}{2}$ cycles. Overall, the number of edges we have deleted is at most
\begin{align*}
4s n \log n + \sum_{i=1}^r 4 |C_i|^{2 - \frac{1}{9}} & \leq 4sn \log n + 4 \left(\max_i |C_i|^{8/9}\right) \cdot \left(\sum_{i=1}^r |C_i|\right)\\ & \leq 4sn \log n + 4 n^{2 - \frac{1}{9}} \leq 16 n^{2 - \frac{1}{9}} \log n \leq n^{2 - \frac{1}{10}},
\end{align*}
where the last inequality holds for $n \geq 2^{2000}$. Since the number of cycles used is at most
\[\sum_{i=1}^r \frac{|C_i|}{2} = \frac{n}{2},\]
this completes the proof.\qed

\section{Random graphs}
\label{secrandom}

In this section, we will prove that the Erd\H{o}s-Gallai conjecture holds a.a.s.~in random graphs. We begin with an elementary lemma which says that if every subgraph of a graph $G$  contains relatively long cycles then $G$ can be decomposed into linearly many cycles and edges.

\begin{lemma}
Let $0<\alpha<1$. Suppose $G$ is a graph on $n$ vertices with the property that every subgraph of $G$ with $m \geq 2n$ edges contains a cycle of length at least $\alpha\frac{m}{n}\log^2 \frac{m}{n}$. Then $G$ can be decomposed into at most $6\alpha^{-1}n$ cycles and edges.
\end{lemma}
\begin{proof}
Greedily pull out longest cycles from $G$ until the remaining subgraph has at most $2n$ edges. To go from a subgraph with $m \leq 2^in$ edges to a subgraph with at most $2^{i-1}n$ edges, at most $$\frac{m}{\alpha2^{i-1}\log^2 2^{i-1}} \leq \frac{2n}{\alpha}(i-1)^{-2}$$ cycles are used.
Hence, at most $\sum_{i \geq 2} \frac{2n}{\alpha}(i-1)^{-2} = \frac{\pi^2}{3\alpha}n<4\alpha^{-1}n$ cycles are used. In total, at most $4\alpha^{-1}n+2n \leq 6\alpha^{-1}n$ cycles and edges are used to decompose $G$.
\end{proof}

We next introduce a concept of sparseness which will guarantee the existence of relatively long cycles in a graph and its subgraphs.

\begin{definition}\label{defsparse1}
For $0<\epsilon < 1 \leq \gamma$, we say that a graph $G$ is $(\epsilon,\gamma)$-sparse if for every $1 \leq v \leq |V(G)|$ every induced subgraph on $v$ vertices contains at most $\gamma v^{2-\epsilon}$ edges.
\end{definition}

Note that any subgraph of an $(\epsilon,\gamma)$-sparse graph is also $(\epsilon,\gamma)$-sparse.

\begin{lemma}
Let $G$ be an $(\epsilon,\gamma)$-sparse graph on $n$ vertices with $m$ edges. Then $G$ contains a cycle of length at least $(m/18 \gamma n)^{1/(1-\epsilon)}$.
\end{lemma}
\begin{proof}
Let $G'$ be a subgraph of $G$ with minimum degree at least $m/n$ (it is easy to see that it exists) and let $t$
be the length of the longest cycle in $G'$. Then, by Lemma~\ref{Posa}, there is a set $S$ with $|S| \leq \lceil t/3 \rceil$ and $|N(S)| \leq 2|S|$. Let $U=S \cup N(S)$, so $|U| \leq 3|S| \leq t+ 2$.
Since $G$ is an $(\epsilon,\gamma)$-sparse graph, the number of edges in $U$ is at most $\gamma|U|^{2-\epsilon}$. However, $U$ has at least $\frac{1}{2}\frac{m}{n}|S| \geq \frac{1}{6}\frac{m}{n}|U| $ edges. Hence,
$$\gamma |U|^{2-\epsilon} \geq e(U) \geq \frac{m}{6n}|U|,$$
from which we get $$3t \geq t + 2 \geq |U| \geq \left(\frac{m}{6\gamma n}\right)^{1/(1-\epsilon)},$$
implying the required result.
\end{proof}

From the preceding two lemmas we have the following immediate corollary.

\begin{corollary} \label{bcor}
For each $0< \epsilon < 1 \leq \gamma$, there is $b > 0$ such that if a graph $G$ on $n$ vertices is $(\epsilon,\gamma)$-sparse then it can be decomposed into at most $b n$ cycles and edges.
\end{corollary}

This corollary is already enough to prove the Erd\H{o}s-Gallai conjecture in sufficiently sparse random graphs.

\begin{lemma}\label{gnqlem}
Suppose $0< \epsilon<1/2$ and $\gamma = 2/\epsilon$. Then, for $q \leq n^{-\epsilon}$, $G(n,q)$ is almost surely $(\epsilon,\gamma)$-sparse. Consequently, $G(n,q)$ and its subgraphs can each be decomposed into at most $b n$ cycles and edges, where $b$ depends only on $\epsilon$.
\end{lemma}

\begin{proof}
As the goal is to show that $G(n,q)$ almost surely has the desired properties, we may assume that $n$ is sufficiently large. Consider an induced subgraph of $G(n,q)$ on $v$ vertices, so it has at most $a:={v \choose 2}$ potential edges. The probability that this induced subgraph contains at least $t:=\gamma v^{2-\epsilon}$ edges is at most
$${a \choose t}q^{t} \leq \left(\frac{qea}{t}\right)^{t} \leq \big((2/e)\gamma (n/v)^{\epsilon}\big)^{-t} \leq \big(2 (n/v)^{\epsilon}\big)^{-t} \leq (en/v)^{-2v^{2-\epsilon}}.$$
As there are ${n \choose v} \leq \left(\frac{en}{v}\right)^v$ subsets of order $v$, the probability that there is a subset of order $v$ with at least $t$ edges is at most $$ \left(en/v\right)^v  \left(en/v\right)^{-2 v^{2-\epsilon}} \leq  \left(en/v\right)^{-v^{2-\epsilon}} .$$
Summing over all $1 \leq v \leq n$, we have that almost surely $G(n,q)$ is $(\epsilon,\gamma)$-sparse. Since all subgraphs of an $(\epsilon, \gamma)$-sparse graph are also $(\epsilon, \gamma)$-sparse, Corollary~\ref{bcor} gives the required conclusion.
\end{proof}

In order to deal with dense random graphs, we note the following property of $G(n, q)$.

\begin{lemma}\label{lemcodeg}
For $q = n^{-\epsilon}$ with $0 < \epsilon \leq 1/5$, the following property holds almost surely in $G(n,q)$. Every pair of distinct vertices have at least $\frac{1}{2} q^2 n$ common neighbors.
\end{lemma}

\begin{proof}
For fixed distinct vertices $u$ and $v$, the expected codegree of $u$ and $v$ is $q^2(n-2)$. Therefore, by Chernoff's inequality, the probability that the codegree of $u$ and $v$ is less than $\frac{1}{2} q^2 n < \frac{3}{4} q^2 (n-2)$ is at most $e^{-q^2 (n -2)/32}$. Taking the union bound over all $\binom{n}{2}$ choices of $u$ and $v$ gives the result.
\end{proof}

We will also need the following simple combination of Theorem \ref{Lovaszlem} and Lemma \ref{sparseends} which was already used in the proof of Lemma~\ref{lem:keycomp}.

\begin{corollary}\label{corobv}
Every graph $G$ can be decomposed into at most $n/2$ cycles and paths and $n$ edges so that each vertex is an endpoint of at most $\sqrt{n}$ of the paths.
\end{corollary}

We are now ready to prove Theorem~\ref{thmrandom}, that the Erd\H{o}s-Gallai conjecture holds almost surely in random graphs.

\vspace{2mm}

\noindent {\bf Proof of Theorem \ref{thmrandom}:}
Let $\epsilon = 1/5$ and $q = n^{-\epsilon}$. If $p \leq q$ then the theorem follows from Lemma~\ref{gnqlem}. Otherwise, we partition the edges of $G(n,p)$ into two graphs $G_1$ and $G_2$ by choosing each edge of $G(n,p)$ independently with probability $q/p$ to form $G_1$ and letting $G_2$ be the complement of $G_1$ in $G(n,p)$. The resulting graphs $G_1$ and $G_2$ are isomorphic to a $G(n,q)$ and a $G(n,p-q)$, respectively. We note that the resulting graphs are not independent random graphs but their individual distributions are identical to the required binomial random graphs. We now apply Corollary~\ref{corobv} to edge partition $G(n,p- q)$ into at most $n/2$ cycles and paths and $n$ edges such that each vertex is an endpoint of at most $\sqrt{n}$ paths. We will use the cycles and edges arising from this procedure (at most $3n/2$ of them) in our final partition. For each of the at most $\frac{n}{2}$ paths in our decomposition of $G(n,p- q)$,  we will greedily pick edge-disjoint paths of length $2$ in $G(n,q)$ that connect their endpoints. We will further require that each vertex is used as the middle vertex of at most $\sqrt{n}$ of these paths.

Suppose that we have achieved our goal for the first $i - 1$ paths and we wish to connect the endpoints $u_i$ and $v_i$ of the $i$th path $P_i$. Then $u_i$ and $v_i$ are each endpoints of at most $\sqrt{n}$ paths. They are each also internal vertices of at most $\sqrt{n}$ paths of length $2$ which we added. Therefore, at most $6 \sqrt{n}$ edges incident to $u_i$ or $v_i$ are already used on paths of length $2$ in $G(n,q)$. Finally, we note that a vertex cannot be used as an internal vertex of any future path of length $2$ if it is already the internal vertex of $\lfloor \sqrt{n} \rfloor$ paths of length $2$. Therefore, there are at most $\frac{n/2}{\lfloor \sqrt{n} \rfloor} \leq \sqrt{n}$ further vertices which cannot be used as internal vertices of paths of length $2$. By Lemma~\ref{lemcodeg}, each pair of vertices in $G(n,q)$ has at least $q^2 n/2 = n^{3/5}/2$ paths of length $2$ between them. Of these, at most $7 \sqrt{n}$ cannot be used due to containing an edge which has already been used or an internal vertex which has been used the maximum number of times. Therefore, since $n$ is sufficiently large, we may find the required path. Since every path completes to either a cycle or a pair of cycles, we get at most $n$ additional cycles. We are now left with a subgraph of $G(n,q)$ and a further application of Lemma~\ref{gnqlem} to this subgraph completes the proof.
\qed

\section{Highly connected graphs}
\label{sectcutdense}

In this section, we will prove some preliminary results that we will need in order to establish Theorem \ref{EGlargemin}. The proof of that result, which states that the Erd\H{o}s-Gallai conjecture holds for graphs of linear minimum degree, will be given in the next section.

The following notion of a graph being dense across cuts will be crucial in what follows.

\begin{definition} \label{defdense1}
A graph $G=(V,E)$ is {\it $d$-cut dense} if every vertex partition $V=V_1 \cup V_2$ satisfies $e(V_1,V_2) \geq d|V_1||V_2|$, i.e., $G$ has density at least $d$ across every cut.
\end{definition}

The main result in this section is Theorem \ref{thmcutdense}, which shows that every $d$-cut dense graph on $n$ vertices can be decomposed into $O(n/d)$ cycles and edges.

We first establish a lemma showing that if the neighborhoods of each pair of vertices in a graph can be connected by many short edge-disjoint paths then we can connect given pairs of vertices by short edge-disjoint paths so that no vertex is used many times internally on the paths. Note that here and throughout what follows, the length of a path will count the number of edges in the path.

\begin{lemma}\label{nice0}
Suppose $G=(V,E)$ is a graph with maximum degree $\Delta$ such that every pair $u,v$ of vertices in $V$ have at least $12\max\left(\sqrt{t\ell},r\right)\Delta+3t\ell$ edge-disjoint paths of length at most $\ell$ between $N(u)$ and $N(v)$. Suppose we are given $t$ pairs $(u_i,v_i)$ for $1 \leq i \leq t$ such that no vertex is in more than $r$ of the pairs. Then there are edge-disjoint paths $s_1,\ldots,s_t$ in $G$, each of length at most $\ell + 2$, such that $s_i$ has endpoints $u_i$ and $v_i$ and each vertex is an internal vertex of at most $B:=\lceil \frac{1}{2}\sqrt{t \ell} \rceil$ of the paths.
\end{lemma}
\begin{proof}
We will greedily find the desired paths $s_1,\ldots,s_t$ in order of index. Recall that there is a collection $P_i$ of at least $12\max\left(\sqrt{t\ell},r\right)\Delta + 3t \ell$ edge-disjoint paths from $N(u_i)$ to $N(v_i)$, each of length at most $\ell$. By possibly shortening some of the paths in $P_i$, we may assume that none of them contain $u_i$ or $v_i$. Call a vertex $v$ {\it dangerous} if it is used internally on $B$ of the paths already picked from $G$. Thus, dangerous vertices cannot be used as internal vertices on any of the paths that have not yet been embedded. If there is a path $p \in P_i$ not containing a dangerous vertex and with endpoints $u' \in N(u_i)$ and $v' \in N(v_i)$ such that none of the edges of $p$ nor the edges $(u',u_i)$ and $(v',v_i)$ are in any $s_j$ with $j<i$, then we can take the path $s_i$ from $u_i$ to $v_i$ to consist of $(u_i,u'),p,(v',v_i)$. This would complete the proof. Thus, it suffices to find such a path $p \in P_i$.

As $u_i$ and $v_i$ are internally on at most $B$ of the paths, at most $2B+r$ of the edges containing $u_i$ are already used on paths, and similarly for $v_i$. Thus, at most $2(2B+r)\Delta$ of the paths in $P_i$ contain a vertex $w$ such that the edge $(u_i,w)$ or $(v_i,w)$ is already used on one of the paths. As there are at most $t$ paths, each of length at most $\ell + 2$, there are at most $\frac{t(\ell + 1)}{B} \leq \frac{2 t \ell}{B}$  dangerous vertices. Therefore, at most $\frac{2t\ell}{B}\Delta$ of the paths in $P_i$ contain a dangerous vertex. Finally, fewer than $t(\ell+2)$ edges are on any of the paths $s_j$ with $j<i$. Hence, at least
$$|P_i|-\left(4B+2r+\frac{2t\ell}{B}\right)\Delta-t(\ell+2)>|P_i|-(6\sqrt{t\ell}+2r + 4)\Delta  -t(\ell+2)\geq 0$$
of the paths in $P_i$  do not contain a dangerous vertex, do not contain an edge used in some path $s_j$ with $j < i$ and do not contain a vertex $w$ such that $(u_i,w)$ or $(v_i,w)$ is already an edge picked in some path $s_j$ with $j<i$.  Hence, we can find the desired path $p \in P_i$, completing the proof.
\end{proof}

We now use Corollary \ref{corobv} and Lemma \ref{nice0} to prove the following lemma, which establishes the Erd\H{o}s-Gallai conjecture for graphs having a subgraph with certain properties.

\begin{lemma}\label{nicestep1}
Suppose $G=(V,E)$ is a graph on $n$ vertices with a subgraph $G'=(V,E')$ of maximum degree $\Delta$ with the following properties:
\begin{itemize}
\item every subgraph of $G'$ can be decomposed into at most $bn$ cycles and edges, and
\item for each pair of vertices $u, v$, there are at least $12\sqrt{\ell n}\Delta+2\ell n$ edge-disjoint paths from $N(u)$ to $N(v)$ of length at most $\ell$ in $G'$.
\end{itemize}
Then $G$ can be decomposed into at most $(2+\frac{\ell}{2}+b)n$ cycles and edges.
\end{lemma}
\begin{proof}
By Corollary \ref{corobv}, there is a partition of $E\setminus E'$ into at most $n/2$ cycles and paths
and $n$ edges so that each vertex is an endpoint of at most $r=\sqrt{n}$ of the paths. Arbitrarily order the paths in the partition as $p_1,\ldots,p_t$, where $t \leq n/2$. Let $u_i$ and $v_i$ be the endpoints of the path $p_i$.

By Lemma \ref{nice0}, there are edge-disjoint paths $s_1,\ldots,s_t$ in $G'$, each of length at most $\ell + 2$ with $s_i$ having endpoints $(u_i,v_i)$. The union of $s_i$ and $p_i$ is a closed walk and, since $s_i$ has length at most
$\ell + 2$, the union of $s_i$ and $p_i$ can be decomposed into at most $\ell + 2$ cycles. As the remaining edges in $G'$ can be partitioned into at most $bn$ cycles and edges, this results in a total of at most $n+(\ell + 2)\frac{n}{2}+bn=(2+\frac{\ell}{2}+b)n$ cycles and edges in the decomposition of $G$.
\end{proof}

The main result in this section is the following theorem.

\begin{theorem} \label{thmcutdense}
Every graph on $n$ vertices which is $d$-cut dense can be decomposed into at most $O(n/d)$ cycles and edges.
\end{theorem}

To prove Theorem \ref{thmcutdense}, it suffices to show that $G$ contains a subgraph $G'$ with the properties of Lemma \ref{nicestep1} with $b=O(1)$ and $\ell=O(1/d)$. To this end, we will show that
a random subgraph of an appropriate density almost surely has the desired properties. Let $G_q$ denote the random subgraph of $G$ in which every edge of $G$ is taken in $G_q$ with probability $q$ independently of the other edges. The choice of $q=n^{-\epsilon}$ with $\epsilon=1/4$ makes it so that we may take $G'=G_q$ almost surely. To show this, we need to first establish some auxiliary lemmas.

A version of Menger's theorem \cite{M27} states that between any two disjoint vertex sets $S$ and $T$ of a graph, the maximum number of edge-disjoint paths from $S$ to $T$ is equal
to the minimum edge-cut separating $S$ and $T$. We would further like to guarantee that the many edge-disjoint paths from $S$ to $T$ are short. To do this will require some more information about the graph. In the last section, we introduced a notion of sparseness saying that no subset contains too many edges. This gives a nice expansion property which we showed implies that any subgraph has (in terms of the number of edges and vertices) a relatively long cycle. It will be convenient to introduce a somewhat different notion of sparseness which will be helpful in guaranteeing that there are many short edge-disjoint paths between any two large vertex subsets. Given two disjoint vertex sets $S$ and $T$ in a graph $G$, we will use the notation $d(S, T) = e(S, T)/|S||T|$ for the density of edges between $S$ and $T$.

\begin{definition}\label{defthin}
A graph $G$ on $n$ vertices is $(q,\rho)$-thin if any disjoint subsets $S,T$ with $|S|,|T| \geq \rho n$ satisfy $d(S,T) \leq q$.
\end{definition}

To show that a graph on $n$ vertices is $(q,\rho)$-thin, it suffices to show that $d(S,T) \leq q$ for any disjoint subsets $S,T$ of cardinality exactly $\lceil \rho n \rceil$. Indeed, if $S$ and $T$ are disjoint sets of cardinality at least $\rho n$ with $d(S,T) > q$, then there are $S' \subset S$ and $T' \subset T$ of cardinality $\lceil \rho n \rceil$ with $d(S',T') \geq d(S,T) > q$. This follows by averaging over all choices of $S'$ and $T'$.

The key lemma we will need about $(q, \rho)$-thin graphs is that if such a graph is also $dq$-cut dense for some $0 < d < 1$ then every pair of large sets has many short edge-disjoint paths between them.

\begin{lemma}\label{shortpath}
Let $0<d,\rho<1$ with $d \geq 4\rho$ and let $\ell=2 \lceil 2/d \rceil$. Let $G=(V,E)$ be a graph which is $dq$-cut dense and $(q,\rho)$-thin. For all disjoint $S,T \subset V$ with $|S|,|T| \geq \rho n$, there are at least $x=\frac{\rho dqn^2}{4\ell}$ edge-disjoint paths of length at most $\ell$  between $S$ and $T$.
\end{lemma}
\begin{proof}
Delete a maximal collection of edge-disjoint paths of length at most $\ell$ from $S$ to $T$ and let $G'$ be the resulting subgraph of $G$. By the definition of $G'$, there is no path from $S$ to $T$ in $G'$ of length at most $\ell$.
Suppose for the sake of contradiction that the number of these edge-disjoint paths is at most $x$, so that at most $x\ell$ edges are deleted from $G$ to obtain $G'$. It suffices to show that more than $n/2$ vertices of $G'$ are distance at most $\ell/2$ from $S$, as by symmetry, we would also get more than $n/2$ vertices of $G'$ are distance at most $\ell/2$ from $T$.
This would imply that there is a vertex of distance at most $\ell/2$ to both $S$ and $T$, and hence a path from $S$ to $T$ in $G'$ of length at most $\ell$, a contradiction. So suppose there are at most
$n/2$ vertices of $G'$ at distance at most $\ell/2$ from $S$.

For $1 \leq i \leq \ell/2$, let $N_i$ denote the set of vertices which are at distance at most $i$ from $S$ in $G'$, so $N_0=S$ and $\rho n \leq |N_0| \leq |N_1| \leq \dots \leq |N_{\ell/2}| \leq n/2$.
Since $G$ is $dq$-cut dense, the number of edges in $G$ from $N_i$ to $N_{i+1} \setminus N_i$ satisfies
$$e_G(N_i,N_{i+1} \setminus N_i)=e_G(N_i,V \setminus N_i) \geq dq|N_i||V \setminus N_i| \geq dq|N_i|\frac{n}{2}.$$
Since $G'$ is obtained from $G$ by deleting at most $x\ell$ edges, $$e_{G'}(N_i,N_{i+1} \setminus N_i) \geq dq|N_i|\frac{n}{2}-x\ell \geq \frac{dqn}{4}|N_i|,$$
where we used that $|N_{i}| \geq |N_0| \geq \rho n$. To estimate the size of $N_{i+1} \setminus N_i$, we let $U=N_{i+1} \setminus N_i$ if $|N_{i+1} \setminus N_i| \geq \rho n$ and otherwise let $U$ be any subset of $V \setminus N_i$ containing $N_{i+1} \setminus N_i$ of order $\rho n$. Then $$q|N_i||U|\geq e_G(N_i,U) \geq e_{G'}(N_i,U) \geq e_{G'}(N_i,N_{i+1}\setminus N_i) \geq \frac{dqn}{4}|N_i|,$$  implying $|U| \geq \frac{d}{4}n$. As $d \geq 4\rho$, this implies that $U = N_{i+1} \setminus N_i$ and $|N_{i+1}\setminus N_i| \geq \frac{d}{4}n$.  By induction on $i$, we get
$|N_i| \geq \rho n +\frac{di}{4}n$.
For $i= \lceil \frac{2}{d} \rceil=\frac{\ell}{2}$, we get $|N_i| > n/2$, a contradiction.
\end{proof}

We will now show that a random subgraph of $G$ has the properties that we need to deduce Theorem~\ref{thmcutdense}.

\begin{lemma}\label{gqlem}
Suppose $0< \epsilon<1/2$, $\gamma = 2/\epsilon$ and $d,\rho \geq \frac{64\ln n}{qn}$.  Let $G$ be a graph on $n$ vertices and $q=n^{-\epsilon}$. Then $G_q$ almost surely has the following properties:
\begin{itemize}
\item $G_q$ is  $(\epsilon,\gamma)$-sparse.
\item $G_q$ is $(2q,\rho)$-thin.
\item $G_q$ has maximum degree at most $2qn$.
\item If $G$ is also $d$-cut dense then almost surely $G_q$ is $qd/2$-cut dense.
\end{itemize}
\end{lemma}
\begin{proof}
As the goal is to show that $G_q$ almost surely has the desired properties, we may assume that $n$ is sufficiently large. We first show that $G_q$ is almost surely $(\epsilon,\gamma)$-sparse. Consider an induced subgraph of $G_q$ on $v$ vertices, so it has at most $a:={v \choose 2}$ edges in $G$. The probability that this induced subgraph contains at least $t:=\gamma v^{2-\epsilon}$ edges is at most $${a \choose t}q^{t} \leq \left(\frac{qea}{t}\right)^{t} \leq \left((2/e)\gamma (n/v)^{\epsilon}\right)^{-t} \leq \left(2 (n/v)^{\epsilon}\right)^{-t} \leq (en/v)^{-2v^{2-\epsilon}}.$$
As there are ${n \choose v} \leq \left(\frac{en}{v}\right)^v$ such subsets of order $v$, the probability that there is a subset of order $v$ with at least $t$ edges is at most $$ \left(en/v\right)^v  \left(en/v\right)^{-2 v^{2-\epsilon}} \leq  \left(en/v\right)^{-v^{2-\epsilon}} .$$
Summing over all $1 \leq v \leq n$, we have that almost surely $G_q$ is $(\epsilon,\gamma)$-sparse.

We next show that $G_q$ is almost surely $(2q,\rho)$-thin. Suppose $S$ and $T$ are disjoint vertex subsets of order $t= \lceil \rho n \rceil $. By the Chernoff bound, the probability that the density between $S$ and $T$ is at least $2q$ is at most $e^{-qt^2/3}$. The number of choices for $S$ and $T$ is ${n \choose t}{n-t \choose t} \leq \frac{n^{2t}}{t!^2}$. Hence, the probability that $G_q$ is not $(2q,\rho)$-thin is at most 
$$e^{-qt^2 /3}\frac{n^{2t}}{t!^2} \leq \frac{1}{t!^2}=o(1),$$
where we used that  $t= \lceil \rho n \rceil$ and $\rho \geq \frac{64\ln n}{qn}$. This
shows that $G_q$ is almost surely $(2q,\rho)$-thin.

We next show that $G_q$ almost surely has maximum degree at most $2qn$. Since $G$ has maximum degree at most $n-1$, the expected degree of each vertex in $G_q$ is at most $q(n-1)$. By the Chernoff bound, the probability that a given vertex has degree at least $2qn$ is at most $e^{-q(n-1)/3}<1/n^2$. Hence, summing over all $n$ vertices, almost surely all vertices have degree less than $2qn$.

Finally, we show that if $G$ is $d$-cut dense, then almost surely $G_q$ is $qd/2$-cut dense.  Let $S$ be a vertex subset of order $v \leq n/2$. The edge density in $G$ between $S$ and $V \setminus S$ is at least $d$, and hence the expected number of edges between $S$ and $V \setminus S$ in $G_q$ is at least $dq|S||V \setminus S| = dqv(n-v)$. By the Chernoff bound, the probability that the density between $S$ and $V \setminus S$ is at most $dq/2$ is at most $e^{-dqv(n-v)/8} \leq e^{-dqvn/16}$. Summing over all choices of $v$ and all ${n \choose v}$ choices of subsets of order $v$,
using that $d \geq \frac{64\ln n}{qn}$, we have that the probability that $G_q$ is not $qd/2$-cut dense is at most
$$\sum_{v=1}^{n/2} {n \choose v}e^{-dqvn/16} \leq \sum_{v=1}^{n/2} n^v e^{-dqvn/16}= \sum_{v=1}^{n/2} e^{v(\ln n - dqn/16)} \leq \sum_{v=1}^{n/2}n^{-v}=o(1).$$
Therefore, $G_q$ is almost surely $qd/2$-cut dense.
\end{proof}

The proof of Theorem~\ref{thmcutdense} is now a straightforward combination of our results.

\vspace{2mm}

\noindent {\bf Proof of Theorem \ref{thmcutdense}:} We may assume that $d \geq 1/\log \log n$, as otherwise the theorem follows from Theorem \ref{firstmain}. We will also assume that $n$ is taken sufficiently large. As already discussed, it suffices to show that  $G_q$ with $q=n^{-\epsilon}$ and $\epsilon=1/4$ almost surely has the properties needed for $G'$ in Lemma \ref{nicestep1} with $\ell=O(1/d)$ and $b=O(1)$. Taking $\gamma=8$ and $\rho=dq/8$, we have by Lemma \ref{gqlem} that $G_q$ almost surely is $(\epsilon,\gamma)$-sparse, $(2q,\rho)$-thin, $qd/2$-cut dense and has maximum degree $\Delta$ at most $2qn$. Fix a subgraph $G'=G_q$ with these properties. By Corollary \ref{bcor}, it follows that $G'$ has the property that every subgraph can be decomposed into at most $bn$ cycles and edges, where $b$ is an absolute constant.  By Lemma \ref{shortpath} applied to $G'$ with $d$ replaced by $d/4$, $q$ replaced by $2q$ and $\ell=2 \lceil 8/d \rceil \leq 32/d$, we have that for all disjoint $S,T \subset V$ with $|S|,|T| \geq \rho n$, there are at least $\frac{\rho (d/4)(2q)n^2}{4\ell} \geq 2^{-11} d^3q^2n^2$ edge-disjoint paths of length at most $\ell$  between $S$ and $T$. Since $G'$ is $qd/2$-cut dense, for any pair of vertices $u,v$, we have $|N(u)|,|N(v)| \geq qd(n-1)/2 \geq \frac{qd}{4}n=2\rho n$. We can find disjoint $S \subset N(u)$ and $T \subset N(v)$ with $|S|,|T| \geq \rho n$. Hence, there are at least $2^{-11} d^3q^2n^2>12\sqrt{\ell n}2qn +2\ell n\geq 12\sqrt{\ell n}\Delta+2\ell n$ edge-disjoint paths of length at most $\ell$  between $N(u)$ and $N(v)$. Since $G'$ has the desired properties in Lemma \ref{nicestep1} with $\ell=O(1/d)$ and $b=O(1)$, Theorem \ref{thmcutdense} follows. \qed

\section{Graphs of linear minimum degree}
\label{sectmin}
In this section, we prove Theorem \ref{firstmain}, verifying the Erd\H{o}s-Gallai conjecture for graphs of linear minimum degree.

We begin with the following lemma. It shows that if a graph has large minimum degree and a sparse cut then there is another sparse cut such that the two subgraphs induced by the vertex subsets still have large minimum degree and average degree.

\begin{lemma}\label{keylong}
Let $0<c<1$ and $d \leq c/2$. Let $G=(V,E)$ be a graph on $n$ vertices with minimum degree at least $cn$ and for which there is a vertex partition $V=V_1 \cup V_2$ with $d(V_1,V_2) \leq d$. Then there is another vertex partition $V=U_1 \cup U_2$ such that, for $i=1,2$, the induced subgraph $G[U_i]$ has minimum degree at least $(c-5dc^{-1})|U_i|$ and $e(U_1,U_2) \leq dn^2$. Furthermore, for $i = 1,2$, the average degree of the vertices in $U_i$ is at most $4c^{-1}dn$ less in $G[U_i]$ than in $G$.
\end{lemma}
\begin{proof}
Suppose without loss of generality that $|V_1| \leq |V_2|$, so $|V_2| \geq n/2$. We have $|V_1| \geq cn/2$. Otherwise, as each vertex has degree at least $cn$, each vertex in $V_1$ has more than $cn/2$ neighbors in $V_2$, and hence $d(V_1,V_2)>c/2$, contradicting $d(V_1,V_2) \leq d$.

For $i \in \{1,2\}$, let $X_i$ be the set of vertices in $V_i$ which have fewer than $c|V_i|$ neighbors in $V_{i}$. As each vertex has at least $cn$ neighbors, each vertex in $X_i$ has at least $cn-c|V_i|=c|V_{3-i}|$ neighbors in $V_{3-i}$. Let $U_i=(V_i\setminus X_i) \cup X_{3-i}$, so $V=U_1 \cup U_2$ is a bipartition of $V$. As
$$d|V_1||V_2| \geq e(V_1,V_2) \geq e(X_i,V_{3-i}) \geq |X_i| \cdot c |V_{3-i}|,$$ we have
$|X_i| \leq dc^{-1}|V_i|$. Hence, \begin{eqnarray*}e(U_1,U_2) & \leq &e(V_1,V_2)+e(X_1,V_1)+e(X_2,V_2) \leq d|V_1||V_2|+|X_1|\cdot c |V_1|+|X_2| \cdot c |V_2| \\ & \leq & d|V_1||V_2|+dc^{-1}|V_1| \cdot c |V_1|+dc^{-1}|V_2| \cdot c |V_2| = d(|V_1||V_2|+|V_1|^2+|V_2|^2) \leq dn^2.
\end{eqnarray*}

We also have $$|U_i| \geq \left(1-dc^{-1}\right)|V_i| \geq \frac{|V_i|}{2} \geq cn/4.$$

As $e(U_1,U_2) \leq dn^2$, the vertices in $U_i$ on average have degree at most $dn^2/|U_i| \leq 4c^{-1}dn$ less in $G[U_i]$ than in $G$.

Note that each vertex in $U_i$, whether in $X_{3-i}$ or $V_i \setminus X_i$, has degree at least $c|V_{i}|$ in $V_{i}$. 
Since $|U_i| \geq cn/4$ and $|V_{3-i}|\leq n$ we also have $|V_{3-i}|\leq 4c^{-1}|U_i|$.
Thus, the minimum degree in $G[U_i]$ is at least
\begin{eqnarray*} c|V_i|-|X_i| & \geq & c|V_i|-dc^{-1}|V_i| = (c-dc^{-1})|V_i| \geq  (c-dc^{-1})\left(|U_i|-|X_{3-i}|\right) \\ & \geq &   (c-dc^{-1})\left(|U_i|-dc^{-1}|V_{3-i}|\right) \geq
 (c-dc^{-1})\left(|U_i|-4dc^{-2}|U_{i}|\right) \\ & = & (c-dc^{-1})(1-4dc^{-2})|U_i| \geq  (c-5dc^{-1})|U_i|,\end{eqnarray*}
 completing the proof.
\end{proof}

We use the preceding lemma to obtain a vertex partition of a graph of linear minimum degree into a bounded number of highly connected vertex subsets.

\begin{lemma}\label{lempartcutdense}
Let $G=(V,E)$ be a graph on $n$ vertices with minimum degree at least $cn$ and $d \leq c^3/80$. Then there
is a vertex partition of $G$ into parts $V_1,\ldots,V_r$ each of order more than $cn/2$ (so $r \leq 2c^{-1} $) such that for each $i$ the induced subgraph $G[V_i]$ has minimum degree at least $\frac{c}{2}|V_i|$ and is $d$-cut dense.
\end{lemma}
\begin{proof}
We produce the partition in steps. In step $i$, we have a
partition $P_i$ of the vertex set $V$ into sets $W_1,...,W_i$ such that
each $G[W_j]$ has minimum degree at least $c(1-20dc^{-2})^{i-1}|W_j|$ and average degree at least $cn-8(i-1)c^{-1}dn$.

Note that this process must terminate within $r=2c^{-1}$ steps. If the process were to run for more than $r$ steps then, for each part in the partition, the induced subgraph on that part would have average degree at least $cn-8rc^{-1}dn > cn/2$. This would in turn imply that each part had more than $cn/2$ vertices, contradicting the fact there are at least $2 c^{-1}$ parts.

In the base case $i=1$, the trivial partition $P_1=\{V\}$ easily has the desired properties above. The induction hypothesis is that we already have a partition $P_i$ with the desired properties. If each part $W_j$ of $P_i$ has the property that $G[W_j]$ does not have a cut with density less than $d$ then $P_i$ is the desired partition for the lemma and we are done. Otherwise, there is a part $W_j \in P_i$  such that $G[W_j]$ has minimum degree $\delta|W_j|$ with $\delta \geq c(1-20dc^{-2})^{i-1} \geq c(1-\frac{c}{4})^{2/c} \geq \frac{c}{2}$, average degree $D \geq cn-8(i-1)c^{-1}dn$ and a cut with density at most $d$. Applying Lemma \ref{keylong}, there is a partition $W_j=U_1 \cup U_2$ such that, for $i=1,2$, the graph $G[U_i]$ has minimum degree at least
$$(\delta -5d \delta^{-1})|U_i| = \delta(1-5d\delta^{-2})|U_i|\geq \delta(1-20dc^{-2})|U_i| \geq c(1-20dc^{-2})^{i}|U_i|$$
and average degree at least
$$D-4\delta^{-1}d|W_j| \geq D-8c^{-1}dn \geq cn-8(i-1)c^{-1}dn-8c^{-1}dn=cn-8ic^{-1}dn.$$
We obtain the partition $P_{i+1}$ from $P_i$ by replacing the part $W_j$ by the two parts $U_1$ and $U_2$. The induction hypothesis and the above analysis shows that $P_{i+1}$ has the desired properties, completing the proof.
\end{proof}

The next lemma has similarities with Lemma \ref{sparseends}. In both of these lemmas, we obtain a decomposition into many paths and edges so that no vertex is an endpoint of too many paths. Whereas Lemma \ref{sparseends} has a better upper bound on the number of paths each vertex is on, here we will be able to guarantee that the ends lie in one part of a bipartite graph.

\begin{lemma}\label{bipdecomp}
Let $H$ be a bipartite graph with $n$ vertices and parts $A$ and $B$ and let $s$ be a positive integer. There is a partition of the edge set of $H$ into at most $n/2$ cycles and paths and $2sn$ edges so that the two endpoints of every path lie in $A$ and no vertex is an endpoint of more than $n/s$ of the  paths.
\end{lemma}

\begin{proof}
Decompose $H$ into at most $n/2$ cycles and paths using Theorem \ref{Lovaszlem}.  For each path that does not have both of its endpoints in $A$, delete one or two of its end edges so that both of its end vertices lie in $A$. We therefore obtain a decomposition of $H$ into cycles, paths and edges with at most $n/2$ cycles and paths and at most $n$ edges such that all the endpoints of the paths lie in $A$. Each path with at most $4(s-1)$ edges we decompose into edges, so the remaining paths have more than $4(s-1)$ edges.

For each path $p$ in the decomposition, let $p_0,p_1$ denote the endpoints of $p$. Let $t \leq n/2$ denote the number of paths $p$ in the decomposition, so the set $D$ of endpoints of paths in the decomposition has $|D|=2t$. We make an auxiliary bipartite graph $L$ with parts $D$ and $A$. In $L$, we connect a vertex $a \in A$ to an endpoint $p_i \in D$ of a path $p$ if $a$ is one of the last $s$ points of $A$ on the end $p_i$ of $p$. Since each path in the decomposition has at least $4s-3$ edges, we have that every vertex in $D$ has at least $s$ neighbors in $A$. We can greedily find a mapping $f:D \rightarrow A$ so that $f$ maps every end in $D$ to one of its at least $s$ neighbors in $L$ and the preimage $f^{-1}(a)$ of every vertex $a \in A$ has size at most $2t/s \leq n/s$.

For each path $p$ in the decomposition, shorten it at both ends so that its new endpoints are $f(p_0)$ and $f(p_1)$. For each path, at most $4(s-1)$ of its edges are deleted. We thus obtain a decomposition into at most $n/2$ cycles and paths and at most $n+4(s-1)\frac{n}{2} \leq 2sn$ edges so that the endpoints of each of the paths are in $A$ and no vertex is an endpoint of more than $n/s$ of the paths.
\end{proof}

We next present the proof of Theorem \ref{EGlargemin}.

\vspace{2mm}

\noindent {\bf Proof of Theorem \ref{EGlargemin}:} We may assume that $n \geq 2^{81}c^{-25}$ as otherwise the result follows from Theorem \ref{firstmain}. Let $d=c^3/80$. By Lemma \ref{lempartcutdense}, there is a vertex partition of $G$ into parts $V_1,\ldots,V_j$ each of order more than $cn/2$ (so $j \leq 2c^{-1} $) such that for each $i$ the induced subgraph $G[V_i]$ has minimum degree at least $\frac{c}{2}|V_i|$ and is $d$-cut dense.

Let $s=2^{40}c^{-11}$.  Let $B_i$ be the bipartite graph induced by parts $V_i$ and $\bigcup_{h>i} V_h$. By Lemma \ref{bipdecomp}, there is a partition of the edge set of $B_i$ into at most $n/2$ cycles and paths and $2sn$ edges so that the two endpoints of every path lie in $V_i$ and no vertex is the endpoint of more than $r=n/s$ paths. Let $p_1,\ldots,p_t$ denote the $t \leq n/2$ paths in the partition and let $u_h,v_h$ be the endpoints of $p_h$.

As $G[V_i]$ is $d$-cut dense, each vertex in $G[V_i]$ has degree at least $d|V_i| \geq dcn/2$. It follows that for every pair $u,v \in V_i$, there are disjoint subsets $S \subset N(u), T \subset N(v)$ of $V_i$ each of cardinality at least $dcn/4$.

By Lemma \ref{shortpath}, as $G[V_i]$ is $d$-cut dense and trivially $(1,dc/4)$-thin, with $\ell=2 \lceil 2/d \rceil \leq 8/d$, there are at least $\frac{(dc/4)d|V_i|^2}{4\ell} \geq 2^{-8}d^3c^2n|V_i|$ edge-disjoint paths of length at most $\ell$  between $S$ and $T$. 
Note that $G[V_i]$ has maximum degree less than $|V_i|$ and
$2^{-8}d^3c^2n|V_i| \geq 3t\ell+12\max\left(\sqrt{t\ell},r\right)|V_i|$, where we used that 
$r=n/s, s=2^{40}c^{-11}, t \leq n/2, \ell \leq 8/d=640c^{-3}, |V_i|\geq cn/2$ and $n \geq 2^{81}c^{-25}$.
Thus, by Lemma \ref{nice0}
there are edge-disjoint paths $s_1,\ldots,s_t$ in $G[V_i]$ where $s_h$ has endpoints $u_h,v_h$ and each vertex is an internal vertex on at most $\lceil \frac{1}{2}\sqrt{t \ell} \rceil$ of these $t$ paths. As $p_h$ and $s_h$ are paths with the same endpoints and $s_h$ has length at most $\ell + 2$, the union of $p_h$ and $s_h$  can be decomposed into at most $\ell+ 2 < 4 \ell$ cycles. This gives a partition of the edge set of $B_i$ into $4\ell t \leq 2\ell n$ cycles and $2sn$ edges.

Let $H_i$ be the subgraph of $G[V_i]$ formed by deleting the edges from $s_1,\ldots,s_t$. We claim that $H_i$ is $d/4$-cut dense. It suffices to show that $d_{H_i}(S,V \setminus S) \geq d/4$ holds for any $S$ with $|S| \leq |V_i|/2$. We split the proof into two cases, depending on whether or not $|S| \leq d|V_i|/4$. For any vertex $v \in G[V_i]$, the number of edges which contain $v$ and are contained in one of the paths $s_h$ is at most  $r+2\lceil \frac{1}{2}\sqrt{t \ell} \rceil < (n/s)+\sqrt{(n/2)(8/d)}+2 \leq cdn/8 \leq d|V_i|/4$ (recall that $s=2^{40}c^{-11}$ and $n \geq 2^{81}c^{-25}$). 
Hence, the  minimum degree in $H_i$ is at least $d(|V_i|-1)-d|V_i|/4 \geq d|V_i|/2$.  Thus, for any subset $S \subset V_i$ with $|S| \leq d|V_i|/4$, each vertex in $S$ has at least $d|V_i|/2-|S| \geq d|V_i|/4$ neighbors in $V_i \setminus S$. It follows that $d_{H_i}(S,V_i \setminus S) \geq d/4$ in this case. The total number of edges on any of the paths $s_1,\ldots,s_t$ is at most $t(\ell + 2) \leq 3t\ell$. Hence, for any $S \subset V_i$ with $d|V_i|/4 \leq |S| \leq |V_i|/2$, using that $|V_i|  \geq cn/2$, $d=c^3/80$  and $n \geq 2^{81}c^{-25}$, we have
$$3t \ell \leq 3 (n/2)(8/d) =12n/d \leq 2^{-6}c^2d^2n^2 = (d/2)(dcn/8)(cn/4)\leq \frac{d}{2}|S||V_i \setminus S|$$
and hence $d_{H_i}(S,V_i \setminus S) \geq d_G(S,V_i \setminus S)-\frac{d}{2} \geq \frac{d}{2}$.

We thus have a decomposition of $G$ into at most $2 j \ell n$ cycles, $2jsn$ edges and the vertex-disjoint graphs $H_1,\ldots,H_j$, each of which is $d/4$-cut dense.  By Theorem \ref{thmcutdense}, each $H_i$ can be decomposed into $O(|H_i|/d)$ cycles and edges. Thus $G$ can be edge-partitioned into at most $2 j\ell n+2jsn+O(\sum_{i=1}^j |H_i|/d)=(2j\ell+2js+O(1/d))n=O(c^{-12}n)$ cycles and edges, completing the proof.
\qed

\section{Concluding remarks}

A conjecture of Haj\'os (see~\cite{L68}) states that every Eulerian graph on $n$ vertices can be decomposed into at most $\frac{n}{2}$ cycles. It is not hard to see that this conjecture implies the Erd\H{o}s-Gallai conjecture. To see this, remove cycles until we are left with a tree. Then the union of the cycles forms an Eulerian graph which, if Haj\'os' conjecture is correct, can be decomposed into at most $\frac{n}{2}$ cycles. Since there are at most $n-1$ edges in the remaining tree, the Erd\H{o}s-Gallai conjecture would follow. Moreover, the example mentioned by Erd\H{o}s~\cite{E83} would imply that the resulting bound of roughly $\frac{3n}{2}$ cycles and edges is asymptotically tight.

If we rephrase this conjecture as asking whether every graph on $n$ vertices can be decomposed into at most $\frac{n}{2}$ cycles and at most $n - 1$ edges then Lemma~\ref{lem:main} may be considered as some small progress. Indeed, this lemma easily implies that any graph on $n$ vertices may be decomposed into at most $\frac{n}{2}$ cycles and $O(n^{2 - \frac{1}{10}})$ edges. This result is the key component in our proof of Theorem~\ref{firstmain} and it would be interesting to know whether it can be sufficiently strengthened to give a further improvement over the bound $O(n \log \log n)$. Progress of a different sort was obtained by Fan \cite{F03}, who proved a covering version of the Haj\'os conjecture, confirming a conjecture of Chung \cite{C80}.

\vspace{2mm}
{\bf Acknowledgements.} We would like to thank David Ellis and Daniel Kane for helpful remarks. We would also like to thank the anonymous referees for a number of useful comments.

\end{document}